\documentclass{amsart}
\usepackage{amsmath,amssymb,amsthm,amsfonts,amscd,mathrsfs,mathtools}
\usepackage{tikz}
\usepackage{braids}
\usepackage[all]{xy}
\usepackage[hyphens]{url}
\usepackage{caption}

\newcommand{\thmref}[1]{Theorem~\ref{#1}}
\newcommand{\propref}[1]{Proposition~\ref{#1}}
\newcommand{\lemref}[1]{Lemma~\ref{#1}}
\newcommand{\corref}[1]{Corollary~\ref{#1}}

\theoremstyle{plain}
    \newtheorem{thm}{Theorem}[section]
    \newtheorem{lem}[thm]   {Lemma}
    \newtheorem{cor}[thm]   {Corollary}
    \newtheorem{prop}[thm]  {Proposition}

\theoremstyle{definition}
    \newtheorem{defn}[thm]  {Definition}

    \newtheorem{ex}[thm]{Example}
    \newtheorem{rem}[thm]{Remark}

\def\cat{\mathsf{cat}}

\newcommand{\be}{\begin{enumerate}}
\newcommand{\ee}{\end{enumerate}}
\newcommand{\R}{\mathbb{R}}
\newcommand{\Z}{\mathbb{Z}}
\newcommand{\C}{\mathbb{C}}
\newcommand{\Q}{\mathbb{Q}}

\newcommand{\TC}{{\sf TC}}

\newcommand{\cld}{{\sf cd}}

\begin{document}

\title[Topological complexity of subgroups of braid groups]{Topological complexity of subgroups of Artin's braid groups}

\author{Mark Grant}
\author{David Recio-Mitter}

\address{Institute of Mathematics,
Fraser Noble Building,
University of Aberdeen,
Aberdeen AB24 3UE,
UK}

\email{mark.grant@abdn.ac.uk}

\email{david.reciomitter@abdn.ac.uk}

\date{\today}


\begin{abstract} We consider the topological complexity of subgroups of Artin's braid group consisting of braids whose associated permutations lie in some specified subgroup of the symmetric group. We give upper and lower bounds for the topological complexity of such mixed braid groups. In particular we show that the topological complexity of any subgroup of the $n$-strand braid group which fixes any two strands is $2n-3$, extending a result of Farber and Yuzvinsky in the pure braid case. In addition, we generalise our results to the setting of higher topological complexity.
\end{abstract}


\maketitle
\section{Introduction}\label{sec:intro}

Topological complexity is a numerical homotopy invariant, introduced by Farber in the course of his topological study of the robot motion planning problem \cite{Far03, Far06}. For any space $X$, the number $\TC(X)$ is defined to be the sectional category of the free path fibration on $X$, and as such gives a quantitative measure of the complexity of navigation in $X$. Computation of $\TC(X)$ for a given space $X$ can be delicate, but is often achievable by combining cohomological lower bounds (in terms of the zero-divisors cup-length \cite[Theorem 7]{Far03}) with upper bounds coming from obstruction theory or the specific geometry of the space at hand.

A class of spaces for which the computation of topological complexity presents a unique challenge are the Eilenberg--Mac Lane spaces $K(\pi,1)$, for $\pi$ a torsion-free discrete group. A description of $\TC(\pi)\coloneqq \TC(K(\pi,1))$ in terms of algebraic properties of the group $\pi$ (as requested by Farber in \cite{Far06}) seems to be out of reach at present. In certain cases one can often compute the exact value of $\TC(\pi)$ using the bounds mentioned above. There are standard bounds $\cld(\pi)\leq \TC(\pi)\leq 2\, \cld(\pi)$ in terms of the cohomological dimension of $\pi$. There are also sharper bounds in terms of cohomological dimensions of certain subgroups or quotient groups of $\pi\times \pi$ (see \cite[Theorem 1.1]{GLO} and \cite[Proposition 3.7]{G}, reproduced as Theorems \ref{thm:lowerbound} and \ref{thm:upperbound} below) which can be of use when the cohomological or obstruction-theoretic bounds are either insufficient, or computationally infeasible.

In this paper we investigate the topological complexity of certain subgroups of Artin's braid groups. Recall that the full braid group $B_n$, the pure braid group $P_n$ and the symmetric group $\mathfrak{S}_n$ fit into an extension
\[
\xymatrix{
1 \ar[r] & P_n \ar[r] & B_n \ar[r]^-{\pi} & \mathfrak{S}_n \ar[r] & 1}
\]
where the projection $\pi$ sends a braid to the associated permutation of its endpoints (definitions will be given in Section \ref{Braidgroups} below). Given any subgroup $G\le \mathfrak{S}_n$, its pre-image $B_n^G\coloneqq\pi^{-1}(G)$ is a subgroup of $B_n$ containing $P_n$. The cohomological dimension of $B_n^G$ is $n-1$, and so by the standard bounds mentioned above we have $n-1\le \TC(B_n^G)\le 2n-2$. We will prove the following.

\begin{thm}\label{thm:introupper}
Suppose that $G\leq \mathfrak{S}_n$ satisfies either of the following conditions:
\begin{itemize}
\item $G\leq\mathfrak{S}_{n-k}\times \mathfrak{S}_k$ where $(n,k)=(n-1,k)=(n-1,k-1)=1$, or
\item $G\leq \mathfrak{S}_{n-2}\times \{1\}^2$.
\end{itemize}
Then we have
\[
\TC(B_n^G)\leq 2n-3.
\]
\end{thm}

\begin{thm}\label{thm:introlower}
Let $G\le \mathfrak{S}_{n-k}\times \mathfrak{S}_k$ for $k\ge2$. Then \[\TC(B_n^G)\ge 2n-k-1.\]

Furthermore, if $G\le \mathfrak{S}_{n-1}\times \{1\}$, then \[\TC(B_n^G)\ge 2n-3.\]
\end{thm}

\begin{cor}\label{exact}
If $G\leq \mathfrak{S}_{n-2}\times \{1\}^2$, then
\[
\TC(B_n^G)=2n-3.
\]
\end{cor}

These results extend the computation of the topological complexity of pure braid groups (due to Farber and Yuzvinsky \cite{FY}) to various \emph{mixed} or \emph{coloured braid groups}. Our methods are somewhat different (in particular we use \cite[Theorem 1.1]{GLO} instead of zero-divisors cup-length).

The above results generalise to the setting of higher topological complexity. The generalised results are formally very similar and thus we do not give them in the introduction. They are discussed in Section \ref{highertc}. We should mention that the results for higher topological complexity extend the computation of the higher topological complexity of pure braid groups due to Gonz\'alez and Grant \cite{GG}. Again, our methods are somewhat different to theirs, which are in turn similar to the ones used by Farber and Yuzvinsky in \cite{FY}.

Configuration spaces of points in the plane give models for Eilenberg--Mac Lane spaces of braid groups, including those considered in this paper (see \lemref{lem:keypione} below). Therefore, our results have implications for motion planning of $n$ agents moving in a planar region avoiding collisions, where the agents are partitioned into equivalence classes according to their function.

Several of the results in this paper were inspired by corresponding results in \cite{CGJ}; we thank the authors of that paper for their correspondence and in particular for providing an algebraic proof of \lemref{lem:centre}. We also thank Jes\'us Gonz\'alez for suggesting that we generalise our results to higher topological complexity.

\section{Topological complexity of aspherical spaces}

For any space $X$, let $p_X:X^I\to X\times X$ denote the free path fibration on $X$, with projection $p_X(\gamma) = (\gamma(0),\gamma(1))$. Recall that the \emph{topological complexity} of $X$, denoted $\TC(X)$, is defined to be the minimal $k$ such that $X\times X$ admits a cover by $k+1$ open sets $U_0,U_1,\ldots , U_k$, on each of which there exists a local section of $p_X$ (that is, a continuous map $s_i:U_i\to X^I$ such that $p_X\circ s_i = \mathrm{incl}_i:U_i\hookrightarrow X\times X$). Note that here we use the reduced version of $\TC(X)$, which is one less than the number of open sets in the cover.

Let $\pi$ be a discrete group. It is well-known that there exists a connected CW-complex $K(\pi,1)$ with
\[\pi_i(K(\pi,1))=\left\{\begin{array}{ll} \pi & (i=1) \\ 0 & (i\ge2). \end{array}\right.\]
Such a space is called an \emph{Eilenberg--Mac Lane space} for the group $\pi$. Furthermore, $K(\pi,1)$ is unique up to homotopy. This makes the following definition sensible.

\begin{defn}
The topological complexity of a discrete group $\pi$ is given by
\[\TC(\pi)\coloneqq\TC(K(\pi,1)).\]
\end{defn}

In the survey article \cite{Far06} Farber poses the problem of describing $\TC(\pi)$ solely in terms of algebraic properties of the group $\pi$. Very little is known about this problem in general. We remark that the corresponding question about $\cat(\pi):=\cat(K(\pi,1))$ has been completely answered: By work of Eilenberg--Ganea \cite{EG} and Stallings \cite{Sta} and Swan \cite{Swa}, we have $\cat(\pi)=\cld(\pi)$, where $\cld$ denotes the cohomological dimension. In cases where the exact value of $\TC(\pi)$ is known, it often agrees with the standard cohomological lower bound in terms of the zero-divisors cup-length. However, one can also give potentially sharper lower bounds for $\TC(\pi)$ which take into account the subgroup structure of $\pi$.

\begin{thm}[{Grant--Lupton--Oprea \cite[Theorem 1.1]{GLO}}]\label{thm:lowerbound}
Let $\pi$ be a discrete group, and let $A$ and $B$ be subgroups of $\pi$. Suppose that $gAg^{-1} \cap B = \{1\}$
for every $g \in \pi$. Then \[\TC(\pi)\ge \cld(A \times B).\]
\end{thm}

This result has been applied in \cite{GLO} to calculate the topological complexity of Higman's group. In Section \ref{Lower bounds} of the current paper we will use it to give lower bounds for the topological complexity of mixed braid groups.

In fact, the upper bounds used in this paper are also group-theoretic in nature.

\begin{thm}[{Grant \cite[Proposition 3.7]{G}}]\label{thm:upperbound}
Let $\pi$ be a torsion-free discrete group, with centre $\mathcal{Z}(\pi)\le \pi$. Identify $\mathcal{Z}(\pi)$ with its image under the diagonal homomorphism $d:\pi\to \pi\times \pi$. Then \[\TC(\pi)\le \cld\left(\frac{\pi\times\pi}{\mathcal{Z}(\pi)}\right).\]
\end{thm}

This result was applied in \cite{G} to give upper bounds for the topological complexity of finitely generated torsion-free nilpotent groups. In Section \ref{Upper bounds} below we will apply it to various subgroups of the full braid group.

\section{Braid groups and their subgroups}\label{Braidgroups}

\begin{defn}
A \emph{braid on $n$ strands} is an isotopy class of embeddings of the disjoint union of $n$ intervals in $\R^3$. The embeddings start at a set of $n$ distinct points in the plane $z=0$ and end at the corresponding set of points in the plane $z=1$, and are monotonically increasing in the direction of the $z$-axis. The isotopies are fixed on the boundary. Note that the start and end points of the individual strands are not required to correspond.

The \emph{full braid group} $B_n$ for $n\ge2$ is the group of braids with $n$ strands, where the group operation is given by concatenating braids.
\end{defn}

The group $B_n$ was first studied by Artin \cite{Art}, who showed (among other things) that it is generated by the braids $\sigma_i$ for $i=1,\ldots ,n-1$ which pass the $i$th strand over the $(i+1)$st strand. There is a canonical epimorphism
\[\pi\colon B_n\to\mathfrak{S}_n\]
which sends the generator $\sigma_i$ to the transposition $(i\;i+1)$ in the symmetric group $\mathfrak{S}_n$. Thus a braid $\gamma\in B_n$ gets sent to the associated permutation $\pi(\gamma)\in \mathfrak{S}_n$ of its endpoints.

The kernel $P_n\coloneqq\pi^{-1}(\{1\})$ of this projection is the \emph{pure braid group} and its elements are called \emph{pure braids}. More generally, given a subgroup $G\leq \mathfrak{S}_n$, we denote its pre-image by $B_n^G\coloneqq\pi^{-1}(G)$. Such a subgroup of $B_n$ may be called a \emph{$G$-braid group}, and its elements \emph{$G$-braids}. These are the groups whose topological complexity we are interested in computing.

\begin{ex}\label{ex:mixed}
Given an integer $1\le k\le n-1$, let $G=\mathfrak{S}_{n-k}\times \mathfrak{S}_k\leq \mathfrak{S}_n$ be the group of permutations which preserve the partition of $n$ objects into the first $n-k$ objects and the last $k$ objects. We denote the subgroup $B_n^{\mathfrak{S}_{n-k}\times \mathfrak{S}_k}$ by $B_{n-k,k}$, and refer to it as a \emph{mixed braid group}. One could imagine using two colours to distinguish the first $n-k$ braids from the last $k$ braids (hence groups of this form also are often called \emph{coloured braid groups}).
\end{ex}

In proving algebraic facts about braid groups and their subgroups, it is often easiest to argue topologically using configuration space models for their $K(\pi,1)$'s. Let
$$\C_m:=\C\setminus\{0,1,\ldots , m-1\}$$
denote the complex plane with $m$ punctures. Recall that the \emph{configuration space} of $n$ points on the plane with $m$ punctures is given by
\[F(\C_m,n)\coloneqq\{(x_1,\ldots,x_n)\in(\C_m)^n\,|\,x_i\ne x_j\text{ for }i\ne j\}.\]

\begin{lem}[Fadell--Neuwirth \cite{FN}]\label{lem:fadellneuwirth}
For each $m\ge0$ and $n\ge2$, projection onto the first coordinate gives a locally trivial fibration sequence
\begin{equation}\label{Fadell-Neuwirth}
F(\C_{m+1},n-1)\to F(\C_m,n)\to\C_m.
\end{equation}
Furthermore, this fibration admits a section.
\end{lem}

\begin{lem}\label{lem:keypione}
Any subgroup $G\leq \mathfrak{S}_n$ of the symmetric group acts freely on $F(\C,n)$ by permuting the coordinates. The quotient space $F(\C,n)/G$ under this action is an Eilenberg--Mac Lane space $K(B_n^G,1)$.
\end{lem}

\begin{proof}
It is easily seen from the definitions that $B_n^G$ is the fundamental group of $F(\C,n)/G$.

It is also well known that $F(\C,n)$ is aspherical. This follows from the Fadell-Neuwirth fibrations (\ref{Fadell-Neuwirth}) and induction. Explicitly, note that $F(\C_{n-1},1) = \C_{n-1}$ is homotopic to a wedge of circles and therefore aspherical, and that for $1\le k\le n-1$ we have that $F(\C_{n-k-1},k+1)$ fibres over an aspherical space $\C_{n-k-1}$ with fibre $F(\C_{n-k},k)$.

The spaces $F(\C,n)/G$ are therefore aspherical, because they have $F(\C,n)$ as a covering space.
\end{proof}

Recall that a duality group is a group whose (co)homology satisfies a generalization of Poincar\'e duality \cite{BE}. In more detail, we call a group $\pi$ a \emph{duality group of dimension $n$} if there exists a $\Z\pi$-module $C$ and an element $e\in H_n(\pi;C)$ such that the cap product homomorphism
\[e\cap-\colon H^k(\pi;A)\overset{\cong}{\to}H_{n-k}(\pi;C\otimes A)\]
is an isomorphism for all $k\in \Z$ and all $\Z\pi$-modules $A$. It follows that $\cld(\pi)=n$.

For instance, the fundamental group of a closed aspherical $n$-manifold is a duality group of dimension $n$.


\begin{thm}[Bieri--Eckmann \cite{BE}]\label{thm:duality}
\begin{enumerate}
\item If a torsion-free group $\pi$ has a finite index subgroup $S$ which is a duality group of dimension $n$, then $\pi$ is also a duality group of dimension $n$.
\item If $K$ and $Q$ are duality groups of dimensions $m$ and $n$ respectively which fit into an extension of groups
\[
\xymatrix{
1 \ar[r] & K \ar[r] & \pi \ar[r] & Q\ar[r] & 1,
}
\]
then $\pi$ is a duality group of dimension $m+n$.
\item Non-trivial free groups are duality groups of dimension $1$.
\end{enumerate}
\end{thm}

The following lemma tells us that the groups we are interested in are indeed duality groups.

\begin{lem}\label{lem:dualitygroups}
For every $G\le\mathfrak{S}_n$, the group $B_n^G$ is a duality group of dimension $n-1$.

Furthermore, $\tilde P_{n-2}\coloneqq\pi_1(F(\C_2,n-2))$ is a duality group of dimension $n-2$.
\end{lem}
\begin{proof}
The fact that $P_n=\pi_1 (F(\C,n))$ is a duality group of dimension $n-1$ follows from the same chain of Fadell--Neuwirth fibrations used in the proof of \lemref{lem:keypione}. Indeed, by Lemma \ref{lem:fadellneuwirth} we have for each $1\le k\le n-1$ a split short exact sequence of fundamental groups
 \[\xymatrix{
 1  \ar[r] & \pi_1 (F(\C_{n-k},k))  \ar[r] & \pi_1 (F(\C_{n-k-1},k+1))  \ar[r] & \pi_1 (\C_{n-k-1})  \ar[r] &  1.}
 \]
  Our inductive hypothesis is that $\pi_1 (F(\C_{n-k},k))$ is a duality group of dimension $k$. The induction step follows from Theorem \ref{thm:duality} (2), with the base case $k=1$ given by Theorem \ref{thm:duality} (3).

Observe that in the course of the induction (two steps before the end) we have shown that $\tilde P_{n-2}\coloneqq\pi_1(F(\C_2,n-2))$ is a duality group of dimension $n-2$.

Finally, let $G\le\mathfrak{S}_n$. The group $B_n^G$ is torsion-free because it has a finite-dimensional $K(\pi,1)$. Now the claim follows from \thmref{thm:duality} (1) because $B_n^G$ contains the duality group $P_n$ as a finite index subgroup.
\end{proof}

In order to apply Theorem \ref{thm:upperbound} to the groups $B_n^G$, we need to identify their centres. Note that $P_n\leq B_n^G\leq B_n$. It is well known that $\mathcal{Z}(P_n)$ and $\mathcal{Z}(B_n)$ are infinite cyclic, and are equal for $n\geq 3$.

\begin{lem}\label{lem:centre}
For any subgroup $G\leq \mathfrak{S}_n$, the centre $\mathcal{Z}(B_n^G)$ is infinite cyclic. Furthermore, when $n\ge3$ we have
\[
\mathcal{Z}(P_n)=\mathcal{Z}(B_n^G)=\mathcal{Z}(B_n),
\]
and all of these groups are generated by the full twist
\[\Delta^2=((\sigma_1\ldots\sigma_{n-1})(\sigma_1\ldots\sigma_{n-2})\ldots(\sigma_1\sigma_2)\sigma_1)^2=(\sigma_1\ldots\sigma_{n-1})^n.\]
\end{lem}
\begin{proof}
The case $n=2$ is clear since all the groups will be isomorphic to $\Z$. For the rest of the proof assume $n\ge3$.

It is a standard result that the full twist $\Delta^2$ generates $\mathcal{Z}(P_n)=\mathcal{Z}(B_n)$ (for a proof see for instance \cite[Theorem 1.24]{KT}). Since $\Delta^2$ commutes with everything (and in particular with braids in $B_n^G$) we therefore have $\mathcal{Z}(P_n)\leq \mathcal{Z}(B_n^G)$. It just remains to be shown that $\mathcal{Z}(B_n^G)\leq P_n$, since then it follows easily that $\mathcal{Z}(B_n^G)\leq \mathcal{Z}(P_n)$. In fact, a seemingly stronger statement is true: If a braid $\gamma\in B_n$ commutes with every pure braid, then $\gamma$ is itself a pure braid.

To see this note that it is possible to assign to each pure braid an \emph{ordered} link by closing off the braid, that is, by connecting together both ends of each strand in the braid. This is illustrated in the following picture.

\begin{center}
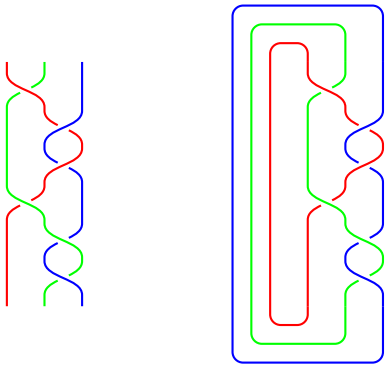

\begin{tikzpicture}[scale=0.5]
\braid[number of strands=3,thick,style strands={1}{red},style strands={2}{green},style strands={3}{blue}] a_1a_2^{-1}a_2^{-1}a_1a_2a_2;
\braid[number of strands=3,thick,style strands={1}{red},style strands={2}{green},style strands={3}{blue},xshift=-8cm] a_1a_2^{-1}a_2^{-1}a_1a_2a_2;
\draw[thick,red,rounded corners](1,0)--(1,0.5)--(0,0.5)--(0,-7)--(1,-7)--(1,-6.5);
\draw[thick,green,rounded corners](2,0)--(2,1)--(-0.5,1)--(-0.5,-7.5)--(2,-7.5)--(2,-6.5);
\draw[thick,blue,rounded corners](3,0)--(3,1.5)--(-1,1.5)--(-1,-8)--(3,-8)--(3,-6.5);
\end{tikzpicture}
\captionof{figure}{A pure braid and its closure.}
\end{center}

Note that conjugating a pure braid by a braid $\gamma\in B_n$ before closing it off will result in the same link, but with the ordering of its components permuted by $\pi(\gamma)$. Suppose that $\gamma\in B_n$ commutes with every braid in $P_n$. Then conjugation by $\gamma$ is the identity on $P_n$, and the permutation $\pi(\gamma)$ acts trivially on the set of isotopy classes of ordered links of $n$ components. It follows that $\pi(\gamma)=1$, so $\gamma$ is a pure braid.
\end{proof}

\section{Upper bounds}\label{Upper bounds}

Using the algebraic results of the previous section, we can now establish our main upper bound.

\begin{thm}\label{thm:torsionbound}
Suppose that $G\le \mathfrak{S}_n$ is such that $B_n^G/\mathcal{Z}(B_n^G)$ is torsion-free. Then
\[\TC(B_n^G)\le 2n-3.\]
\end{thm}

\begin{proof}
Consider the group extension
\begin{equation*}
\xymatrix{ 1 \ar[r] &  \mathcal{Z}(B_n^G) \ar[r]^-{d} &  B_n^G\times B_n^G \ar[r] & (B_n^G\times B_n^G)/\mathcal{Z}(B_n^G) \ar[r] & 1,}
\end{equation*}
where $d$ denotes the diagonal embedding. By Theorem \ref{thm:upperbound}, it will suffice to prove that the quotient group $(B_n^G\times B_n^G)/\mathcal{Z}(B_n^G)$ has cohomological dimension $2n-3$. It follows from the assumptions that this quotient is torsion-free, and by Lemma \ref{lem:centre} it contains $(P_n\times P_n)/\mathcal{Z}(P_n)$ as a subgroup of finite index. Thus we are reduced (by Theorem \ref{thm:duality}(1)) to showing that $(P_n\times P_n)/\mathcal{Z}(P_n)$ is a duality group of dimension $2n-3$.

Using the operations of addition and multiplication in $\C$, it is easily shown that the first two Fadell--Neuwirth fibrations used in the proofs of Lemmas \ref{lem:keypione} and \ref{lem:dualitygroups} are in fact trivial, so that we have a splitting up to homeomorphism
\[F(\C,n)\approx F(\C_2,n-2)\times\C_1\times\C.\]
On fundamental groups this gives a splitting $P_n\cong\tilde{P}_{n-2}\times\Z$, under which the centre $\mathcal{Z}(P_n)$ corresponds to $\{1\}\times\Z$. We therefore find that
\[
(P_n\times P_n)/\mathcal{Z}(P_n) \cong (\tilde{P}_{n-2}\times\Z\times \tilde{P}_{n-2}\times\Z)/\mathcal{Z}(P_n) \cong \tilde{P}_{n-2}\times \tilde{P}_{n-2}\times \Z.
\]
Since the group $\tilde{P}_{n-2}$ was shown in \lemref{lem:dualitygroups} to be a duality group of dimension $n-2$, the result now follows from Theorem \ref{thm:duality}(2).
\end{proof}

We now give examples of subgroups $G\leq \mathfrak{S}_n$ for which the assumption of the theorem holds. We continue to denote the mixed braid group $B_n^{\mathfrak{S}_{n-k}\times \mathfrak{S}_k}$ from Example \ref{ex:mixed} by $B_{n-k,k}$.

\begin{prop}\label{prop:torsionfree}
The mixed braid group modulo its centre $B_{n-k,k}/\mathcal{Z}(B_{n-k,k})$ is torsion-free if and only if $(n,k)=(n-1,k)=(n-1,k-1)=1$.
\end{prop}
\begin{proof}

The centre $\mathcal{Z}(B_{n-k,k})$ is infinite cyclic generated by the full twist $\Delta^2=(\sigma_1\ldots\sigma_{n-1})^n$, by \lemref{lem:centre}. Thus torsion elements in $B_{n-k,k}/\mathcal{Z}(B_{n-k,k})$ will be represented by braids $\alpha\in B_{n-k,k}$ such that $\alpha^m=\Delta^l$, for some powers $m,l\in\Z$, but $\alpha$ is not a power of $\Delta^2$ itself.

Viewing the braid group $B_n$ as the fundamental group of the configuration space $F(\C,n)/\mathfrak{S}_n$, the full twist corresponds to a rotation of the point configuration by $2\pi$. Section~2.2 from \cite{GM} tells us that all torsion elements are represented by rotations up to conjugation. To be precise, there exist $m,l\in\Z$ such that $\alpha^m=\Delta^l$ if and only if there exist $\beta\in B_n$ and $p\in\Z$ such that either $\alpha=\beta\delta^p\beta^{-1}$ or $\alpha=\beta\epsilon^p\beta^{-1}$. Here the elements $\delta$ and $\epsilon$ can be expressed in terms of the standard generators as
\[\delta=\sigma_1\ldots\sigma_{n-1}\text{ and }\epsilon=\sigma_1(\sigma_1\ldots\sigma_{n-1}),\] and correspond to minimal rotations in $F(\C,n)/\mathfrak{S}_n$, where in the case of $\epsilon$ one of the points in the configuration is a fixed point of the rotation.

Recall that $\pi$ denotes the canonical projection $\pi\colon B_n\to\mathfrak{S}_n$.

Observe that $\pi(\delta)$ is an $n$-cycle and $\pi(\epsilon)$ is an $(n-1)$-cycle times a 1-cycle. Therefore $\pi(\delta)^p$ is a product of $d$ $n/d$-cycles, where $d=(p,n)$ is the greatest common divisor of $p$ and $n$, which could be 1. Similarly $\pi(\epsilon)^q$ is a product of $d$ $(n-1)/d$-cycles and a 1-cycle, where $d=(q,n-1)$.

Recall that elements in the symmetric group are conjugate if and only if they have the same cycle structure. Therefore, $\pi(\delta)^p$ is conjugate to an element in $\mathfrak{S}_{n-k}\times\mathfrak{S}_k$ if and only if $n/(p,n)$ divides $k$. Similarly, $\pi(\epsilon)^q$ is conjugate to an element in $\mathfrak{S}_{n-k}\times\mathfrak{S}_k$ if and only if $(n-1)/(q,n-1)$ divides $k$ or $k-1$.

One implication in the theorem now follows from combining the above observations. If $B_{n-k,k}/\mathcal{Z}(B_{n-k,k})$ had torsion there would exist an element $\alpha\in B_{n-k,k}$ which is conjugate to $\delta^p$ or $\epsilon^q$. Therefore, there would exist an element $\pi(\alpha)\in\mathfrak{S}_{n-k}\times\mathfrak{S}_k$ conjugate to $\pi(\delta)^p$ or $\pi(\epsilon)^q$. This implies that either $(n,k)\ne1$ or $(n-1,k)\ne1$ or $(n-1,k-1)\ne1$.

The converse implication follows from a similar argument. Assume that for instance $(n,k)\ne1$. Then there exists an element $\alpha\in\mathfrak{S}_{n-k}\times\mathfrak{S}_k$ conjugate to $\pi(\delta)^p$:
\[\alpha=\beta\pi(\delta)^p\beta^{-1}.\]

Taking a lift $\sigma_\beta$ of $\beta$ yields an element $\sigma_\beta\delta^p\sigma_\beta^{-1}$ of $B_{n-k,k}$.
\end{proof}

\begin{cor}\label{cor:torsion1}
If $G\le \mathfrak{S}_{n-k}\times \mathfrak{S}_k$ with $(n,k)=(n-1,k)=(n-1,k-1)=1$, then the group $B_n^G/\mathcal{Z}(B_n^G)$ is torsion-free.
\end{cor}
\begin{proof}
If $G\le \mathfrak{S}_{n-k}\times \mathfrak{S}_k$, then $B_n^G$ is a subgroup of $B_{n-k,k}$ but by \lemref{lem:centre}, the centre of both groups is the same. Therefore, if $B_{n-k,k}/\mathcal{Z}(B_{n-k,k})$ is torsion-free then so is $B_n^G/\mathcal{Z}(B_n^G)$.
\end{proof}

\begin{prop}\label{prop:torsion2}
If $G=H\times\{1\}^2\le\mathfrak{S}_n$ where $H\le \mathfrak{S}_{n-2}$, then the group $B_n^G/\mathcal{Z}(B_n^G)$ is torsion-free.
\end{prop}

\begin{proof}
In the proof of \thmref{thm:torsionbound} it was observed that there is a homeomorphism $F(\C,n)\approx F(\C_2,n-2)\times\C_1\times\C$. This splitting is compatible with the action of $G=H\times\{1\}^2$ (as is clear from writing out explicit formulae, or see Lemma 3.1 and Theorem 3.3 of \cite{CGJ} for a proof). Therefore we have a homeomorphism $$F(\C,n)/G\approx (F(\C_2,n-2)/H)\times \C_1\times \C,$$ which on fundamental groups gives the splitting $B_n^G\cong \tilde{B}_{n-2}^H\times \Z$, where $\tilde{B}_{n-2}^H\coloneqq\pi_1(F(\C_2,n-2)/H)$. Under this splitting the centre $\mathcal{Z}(B_n^G)=\Z$ corresponds to $\{1\}\times\Z$, and therefore we have $B_n^G/\mathcal{Z}(B_n^G)\cong \tilde{B}_{n-2}^H$.

Finally we observe that $\tilde{B}_{n-2}^H$ is torsion-free because it has the finite-dimensional space $F(\C_2,n-2)/H$ as its $K(\pi,1)$.
\end{proof}

Putting together \corref{cor:torsion1}, \propref{prop:torsion2} and \thmref{thm:torsionbound} yields:

\begin{thm}\label{thm:upperresult}
Let $G\le \mathfrak{S}_{n-2}\times\{1\}^2\le \mathfrak{S}_n$ or $G\le \mathfrak{S}_{n-k}\times \mathfrak{S}_k$ with $(n,k)=(n-1,k)=(n-1,k-1)=1$. Then \[\TC(B_n^G)\le 2n-3.\]
\end{thm}

\begin{rem}
In the case $G\leq \mathfrak{S}_{n-2}\times\{1\}^2$, the upper bound $\TC(B_n^G)\le 2n-3$ can be deduced from the fact that there is a splitting
\[
B_n^G \cong \left(B_n^G/\mathcal{Z}(B_n^G)\right)\times \mathcal{Z}(B_n^G),
\]
together with the product formula for topological complexity \cite[Theorem 11]{Far03}. We do not know if the upper bound $\TC(B_n^G)\le 2n-3$ for $G\leq\mathfrak{S}_{n-k}\times \mathfrak{S}_k$ can be obtained in this way. For instance, we do not know if the centre of the mixed braid group $B_{5,3}$ splits off as a direct factor.
\end{rem}

\section{Lower bounds}\label{Lower bounds}

We now give lower bounds for $\TC(B_n^G)$ in certain cases, using Theorem \ref{thm:lowerbound}. Our arguments generalise the argument given in \cite[Proposition 3.3]{GLO} for pure braids.

\begin{thm}\label{lower1}
Let $G\le \mathfrak{S}_{n-k}\times \mathfrak{S}_k$ for $k\ge2$. Then \[\TC(B_n^G)\ge 2n-k-1.\]

Furthermore, if $G\le \mathfrak{S}_{n-1}\times \mathfrak{S}_1$, then \[\TC(B_n^G)\ge 2n-3.\]
\end{thm}
\begin{proof}

Suppose $G\le \mathfrak{S}_{n-k}\times \mathfrak{S}_k$ for $k\ge2$. To apply \thmref{thm:lowerbound}, we need to give two subgroups of $B_n^G$ which contain no non-trivial conjugates. The following subgroups satisfy this assumption:

\begin{enumerate}
\item Think of $P_{n-k+1}$ for $k\ge 2$ as a subgroup of $P_n$ by adding $k-1$ strands which do not interact with any other strand.
\item Denote by $A_n$ the subgroup of $P_n$ generated by braids of the form \[\alpha_j=\sigma_{j}\sigma_{j+1}\ldots \sigma_{n-1}\sigma_n\sigma_n\sigma_{n-1}\ldots \sigma_{j+1}\sigma_{j},\]
for $1\le j\le n-1$. In words, the generator $\alpha_j$ takes the $j$th strand over the last $n-j$ strands and back under them until it gets back to its original position, see the braid diagrams below. Looking at the braid diagrams it becomes evident that the generators commute, and therefore $A_n\cong\Z^{n-1}$.
\end{enumerate}

Since $P_n\leq B_n^G$, these are indeed both subgroups of $B_n^G$. It remains to show that $P_{n-k+1}\cap \gamma A_n\gamma^{-1}=\{1\}$ for all $\gamma\in B_n^G$.

Recall that closing off a pure braid yields an ordered link, and that conjugating a pure braid in $P_n$ by an element of $B_n^G$ before closing it off will result in an isotopic link up to permutation of its components by an element of $G$. In light of the above, it suffices to show that the links coming from closing off nontrivial elements of $P_{n-k+1}$ can not be obtained by permuting with an element in $G$ the components of the link coming from closing off an element of $A_n$.

 For this we note that if we close a non-trivial braid in $A_n$, the last two components of the link will be linked either with some other component or with each other. This can be seen for the generators $\alpha_i$ by taking a look at the following braid diagrams.

\vspace{.3cm}
\begin{minipage}{.48\linewidth}
\begin{center}
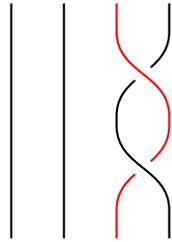

\begin{tikzpicture}[xscale=0.7,yscale=1.25]
\braid[number of strands=4,thick,style strands={3}{red}] 1 a_3a_3 1;
\end{tikzpicture}
\captionof{figure}{The braid $\alpha_{3}\in A_4$. }
\end{center}
\end{minipage}%
\begin{minipage}{.48\linewidth}
\begin{center}
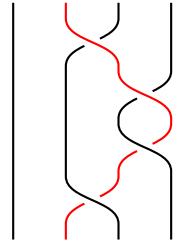

\begin{tikzpicture}[scale=0.7]
\braid[number of strands=4,thick,style strands={2}{red}] a_2a_3a_3a_2;
\end{tikzpicture}
\captionof{figure}{The braid $\alpha_{2}\in A_4$. }
\end{center}
\end{minipage}
\vspace{.3cm}

Because the generators $\alpha_i$ of $A_n$ commute with each other, every element $\sigma \in A_n$ can be written uniquely in the form
\[\sigma=\alpha_{1}^{m_{1}}\alpha_{2}^{m_{2}}\cdots\alpha_{n-1}^{m_{n-1}}\qquad(m_i\in \Z).\]
The exponent $m_i$ determines the linking number of the $i$th strand with all subsequent strands, from the $(i+1)$st up to the $n$th. Therefore the last two components of the link which results from closing off $\sigma$ are unlinked from the first $n-2$ and from each other if and only if all the $m_i=0$, i.e. if and only if $\sigma$ is trivial. In particular, closing off a non-trivial element of $A_n$ will yield a link where at least two of the last $k$ components are linked with something else.

On the other hand, closing braids in $P_{n-k+1}\subset P_{n}$ yields links in which at most one of the last $k$ components is linked with another component. Therefore, a link coming from a non-trivial element in $A_n$ cannot be obtained from a link coming from a braid in $P_{n-k+1}$ by permuting the components by an element of $G$. This is because, by assumption, an element of $G$ does not permute any of the last $k$ components with any of the first $n-k$.

Finally, \thmref{thm:lowerbound} yields

\[\TC(B_n^G)\ge \cld(P_{n-k+1})+\cld(A_n)=(n-k)+(n-1)=2n-k-1.\]

The proof for $G\le \mathfrak{S}_{n-1}\times \mathfrak{S}_1$ is analogous but has to be stated separately. Here we may take the subgroups $P_{n-1}$ and $A_n$, and just use the linking number with the last strand.
\end{proof}

Note that the lower bound goes down by one at each step for $k\ge2$ but it is the same for $B_{n-1,1}$ and $B_{n-2,2}$.

\begin{cor}\label{cor:lower2}
Let $G\le \mathfrak{S}_{n-2}\times\{1\}^2\le \mathfrak{S}_n$. Then \[\TC(B_n^G)\ge 2n-3.\]
\end{cor}

Note that in \corref{cor:lower2} the lower bound coincides with the upper bound from \thmref{thm:upperresult} and thus gives a complete answer in that case.

\begin{rem}\label{rem:zerodivisors}
The optimal lower bound $\TC(P_n)\geq 2n-3$ can be obtained fairly easily using zero-divisors cup-length \cite{FY}, because the rational cohomology algebra $H^*(P_n;\Q)$ is well understood.

One could ask whether the lower bounds for $\TC(B_n^G)$ given above can be obtained similarly using rational cohomology. An application of the Cartan--Leray spectral sequence would give an algebra isomorphism
\[
H^*(B_n^G;\Q)\cong H^*(P_n;\Q)^G.
\]
In the cases we are interested in, however, the calculation of this ring of invariants appears not so straightforward, and does not seem to have been carried out in the literature. In the low dimensional cases in which we were able to compute the zero-divisors cup-length, it was less effective than our lower bound. One could try to replace $\Q$ with some other coefficient ring, but then calculations become even more involved.

For the full braid group $B_n$, both the zero-divisors cup-length and the methods given in this paper appear to be insufficient.
\end{rem}

\section{Higher topological complexity}\label{highertc}

The concept of higher topological complexity was introduced by Rudyak in \cite{Rud}. See also the subsequent paper \cite{BGRT} by Basabe, Gonz\'alez, Rudyak and Tamaki.

Recall that in Section 2 the topological complexity of a space $X$ was defined using the free path fibration $p_X:X^I\to X\times X$, with projection $p_X(\gamma) = (\gamma(0),\gamma(1))$. For any natural number $m\ge2$, the \emph{$m$th topological complexity} of $X$, denoted $\TC_m(X)$, can similarly be defined with help of the fibration $$p^m_X:X^I\to X^m$$ with projection $$p^m_X(\gamma) = (\gamma(0),\gamma(1/(m-1)),\ldots,\gamma((m-2)/(m-1)),\gamma(1)).$$Concretely, it is the minimal $k$ such that $X^m$ admits a cover by $k+1$ open sets $U_0,U_1,\ldots , U_k$, on each of which there exists a local section of the fibration $p^m_X$. Of course, the case $m=2$ corresponds to the original topological complexity.

The results for the topological complexity of mixed braid groups in the last two sections rely on \cite[Theorem 1.1]{GLO} and  \cite[Proposition 3.7]{G} (denoted \thmref{thm:lowerbound} and \thmref{thm:upperbound} in this paper). Those theorems can be generalised in a straightforward way to the following statements about higher topological complexity.

\begin{thm}\label{thm:higherlowerbound}
Let $\pi$ be a discrete group, and let $A$ and $B$ be subgroups of $\pi$. Suppose that $gAg^{-1} \cap B = \{1\}$
for every $g \in \pi$. Then \[\TC_m(\pi)\ge \cld(A \times B \times \pi^{m-2}).\]
\end{thm}

\begin{thm}\label{thm:higherupperbound}
Let $\pi$ be a torsion-free discrete group, with centre $\mathcal{Z}(\pi)\le \pi$. Identify $\mathcal{Z}(\pi)$ with its image under the diagonal homomorphism $d_m:\pi\to \pi^m$. Then \[\TC_m(\pi)\le \cld\left(\frac{\pi^m}{\mathcal{Z}(\pi)}\right).\]
\end{thm}

Using this, the proofs in the earlier sections generalise to yield the following results.

\begin{thm}\label{thm:higherupperresult}
Suppose that $G\leq \mathfrak{S}_n$ satisfies either of the following conditions:
\begin{itemize}
\item $G\leq\mathfrak{S}_{n-k}\times \mathfrak{S}_k$ where $(n,k)=(n-1,k)=(n-1,k-1)=1$, or
\item $G\leq \mathfrak{S}_{n-2}\times \{1\}^2$.
\end{itemize}
Then we have
\[
\TC_m(B_n^G)\leq m(n-1)-1.
\]
\end{thm}

\begin{thm}\label{thm:higherlowerresult}
Let $G\le \mathfrak{S}_{n-k}\times \mathfrak{S}_k$ for $k\ge2$. Then \[\TC_m(B_n^G)\ge m(n-1)-k+1.\]

Furthermore, if $G\le \mathfrak{S}_{n-1}\times \{1\}$, then \[\TC(B_n^G)\ge m(n-1)-1.\]
\end{thm}

As before, in some cases the upper and lower bounds coincide to give an equality.

\begin{cor}
If $G\leq \mathfrak{S}_{n-2}\times \{1\}^2$, then
\[
\TC(B_n^G)=m(n-1)-1.
\]
\end{cor}

\end{document}